\documentclass[11pt,a4paper]{article}

\usepackage[margin=2.5cm]{geometry}
\usepackage{array}
\newcolumntype{x}{>{\vbox to 5ex\bgroup\vfill\centering\arraybackslash\hspace{0pt}}m{2.2cm}<{\egroup}}
\usepackage[fleqn]{amsmath}
\usepackage{amssymb,latexsym}
\usepackage{amsthm}
\usepackage[colorlinks=true,linkcolor=black,citecolor=black,urlcolor=black]{hyperref}
\usepackage{enumitem}
\usepackage{graphicx}
\usepackage{tikz}

\newtheorem{theorem}{Theorem}[section]
\newtheorem{corollary}[theorem]{Corollary}

\newtheorem{proposition}[theorem]{Proposition}

\DeclareMathOperator{\Aut}{Aut}
\DeclareMathOperator{\LS}{LS}
\DeclareMathOperator{\STS}{STS}
\title{On the upper embedding of symmetric configurations with block size 3}

\author{
Grahame Erskine\thanks{Open University, Milton Keynes, UK}\footnotemark[1]\\ \texttt{\small grahame.erskine@open.ac.uk}
\and Terry Griggs\footnotemark[1]\\ \texttt{\small terry.griggs@open.ac.uk}
\and Jozef \v{S}ir\'a\v{n}\thanks{Open University, Milton Keynes, UK and Slovak University of Technology, Bratislava, Slovakia}\footnotemark[2]\\ \texttt{\small jozef.siran@open.ac.uk}
}

\date{}

\begin{document}
\maketitle
\let\thefootnote\relax
\let\thefootnote\relax

\vspace*{-4ex}
\begin{abstract}
\noindent We consider the problem of embedding a symmetric configuration with block size 3 in an orientable surface in such a way that the blocks of the configuration form triangular faces and there is only one extra large face. We develop a sufficient condition for such an embedding to exist given any orientation of the configuration, and show that this condition is satisfied for all configurations on up to 19 points. We also show that there exists a configuration on 21 points which is not embeddable in any orientation. As a by-product, we give a revised table of numbers of configurations, correcting the published figure for 19 points. We give a number of open questions about embeddability of configurations on larger numbers of points.
\end{abstract}

% ================================================================
\section{Introduction}\label{sec:intro}
Previous work~\cite{Grannell2005,Griggs2018,Griggs2019} considered the problem of upper embedding of Steiner triple systems and Latin squares. The main motivation of the current paper is to generalise this idea to the case of symmetric configurations with block size 3.  We begin with the requisite background and definitions and initially work more generally.  Let ${\cal X}=(V,{\cal B})$ be a (partial) triple system on a point set $V$, that is, a collection ${\cal B}$ of $3$-element subsets of $V$, called \emph{blocks} or \emph{triples}, such that every $2$-element subset of $V$ is contained in at most one triple in ${\cal B}$. Equivalently, such a triple system ${\cal X}$ may be viewed as a pair $(K,{\cal B})$, where $K$ is a graph with vertex set $V$ and edge set $E$ consisting of all pairs $uv$ for distinct points $u,v\in V$ such that $\{u,v\}$ is a subset of some block in ${\cal B}$. In other words, in the graph setting, ${\cal B}$ is regarded as a decomposition of the edge set $E$ of $K$ into triangles. Of course, such a graph $K=(V,E)$ may admit many decompositions into triangles and so the set of blocks ${\cal B}$ needs to be specified. We will refer to $K=(V,E)$ with the specified decomposition ${\cal B}$ of $E$ as the graph {\em associated} with the triple system ${\cal X}=(V,{\cal B})$. We will be assuming throughout that the triple systems considered here are {\em connected}, meaning that their associated graphs are connected.

In the case of a Steiner triple system $\mathcal{S}=\STS(n)$ where $n=|V|$, the associated graph is the complete graph $K_n$. Such systems exist if and only if $n \equiv$ 1 or 3$\pmod 6$ \cite{Kirkman1847}. For a Latin square $L = \LS(n)$, the associated graph is the complete tripartite graph $K_{n,n,n}$ where the three parts of the tripartition are the rows, the columns and the entries of the Latin square. 
A \emph{configuration} is a finite incidence structure with $v$ points and $b$ blocks, with the property that there exist positive integers $k$ and $r$ such that:
\begin{enumerate}[label=(\roman*),topsep=0ex,itemsep=-1ex]
	\item each block contains exactly $k$ points;
	\item each point is contained in exactly $r$ blocks; and
	\item any pair of distinct points is contained in at most one block.
\end{enumerate}
If $v=b$ (and hence necessarily $r=k$), the configuration is called \emph{symmetric} and is usually denoted by $v_k$. In this paper we are concerned only with the case of symmetric configurations $v_3$.
Clearly a symmetric configuration $v_3$ is a triple system, and its associated graph is regular of valency 6.

By an \emph{embedding} of a triple system ${\cal X}=(V,{\cal B})$ we will understand a cellular embedding $\vartheta:\ K\to \Sigma$ of the associated graph $K=(V,E)$ of $\cal X$ in an orientable surface $\Sigma$, such that every triangle in ${\cal B}$ bounds a face of $\vartheta$. Such faces will be called \emph{block faces}, and the remaining faces of the embedding $\vartheta$ will be called \emph{outer faces}. By the properties of the set ${\cal B}$, in the embedding $\vartheta$ every edge of $E$ lies on the boundary of at most one block face, so that there is at least one outer face in $\vartheta$. The extreme case occurs if such an embedding has exactly one outer face; we then speak about an \emph{upper embedding} and call the triple system ${\cal X}=(V,{\cal B})$ \emph{upper embeddable}. We note here that in an upper embeddable triple system, Euler's formula $V+F-E=2-2g$ implies that $v+(v+1)-3v=1-v=2-2g$, so that $v$ is necessarily odd.

It was proved in \cite{Grannell2005} that every Steiner triple system $\STS(n)$ and in \cite{Griggs2018} that every Latin square $\LS(n)$ where $n$ is odd has an upper embedding in an orientable surface. In the Latin square case, again, the restriction that $n$ must be odd is determined by Euler's formula. 
The basic method employed in both of these papers is first to embed a subset of the set of triples $\mathcal{B}$ in a sphere with one outer face which contains all the points $V$. Further triples are then added one at a time using an additional handle until the required upper embedding is obtained.

A necessary and sufficient condition for upper embeddability of triple systems follows from available knowledge about upper embeddings of graphs in general. To make use of this we will represent triple systems by their point-block incidence graphs as usual in design theory. For a triple system ${\cal X}=(V,{\cal B})$ its \emph{Levi graph} or \emph{point-block incidence graph} is the bipartite graph $G({\cal X})$ with vertex set $V\cup {\cal B}$ and edge set consisting of pairs $\{v,B\}$ for $v\in V$ and $B\in {\cal B}$ such that $v \in B$. The pair $(V,{\cal B})$ forms the bi-partition of the vertex set of $G({\cal X})$; vertices in $V$ and ${\cal B}$ will be referred to as \emph{point vertices} and \emph{block vertices} respectively. By our convention for symmetric configurations, the graph $G({\cal X})$ is assumed to be connected, and note that every point vertex and every block vertex has valency $3$ in $G({\cal X})$.

As is well known, there is a one-to-one correspondence between orientable embeddings of a triple system $\mathcal{X}$ and its Levi graph $G(\mathcal{X})$ in orientable surfaces. See~\cite{Griggs2019} for an explanation of this. In particular, a triple system $\mathcal{X}$ is upper embeddable if and only if its Levi graph is embeddable with exactly one face; such graphs are also called upper embeddable. By a classical result of Jungerman~\cite{Jungerman1978}, a graph is upper embeddable if and only if the graph contains a spanning tree such that each of its co-tree components has an even number of edges. Stating this formally, we have the following result.
\begin{theorem}[\cite{Jungerman1978}]\label{thm:jungerman}
Let $\mathcal{X}$ be a triple system and let $G=G(\mathcal{X})$ be its Levi graph. Then $\mathcal{X}$ admits an upper embedding if and only if $G$ admits a spanning tree such that each of its co-tree components has an even number of edges.
\end{theorem}
In both~\cite{Grannell2005} and~\cite{Griggs2018}, no attention was paid to the orientation of the triples in the upper embeddings. However in a more recent paper~\cite{Griggs2019}, the question was posed of which triple systems ${\cal X}=(V,{\cal B})$ have the property that for every choice of orientation of all triples $B\in\mathcal{B}$, the system $\mathcal{X}$ admits an upper embedding in an orientable surface such that the orientation of the surface induces the preassigned orientation of $B$ for every triple $B\in\mathcal{B}$. This property was called \emph{upper embeddability in every orientation of triples}. The authors of~\cite{Griggs2019} found the following sufficient condition.
\begin{theorem}[\cite{Griggs2019}]\label{thm:evenvalency}
Let $\mathcal{X}$ be a triple system and let $G=G(\mathcal{X})$ be its Levi graph. If $G$ admits a spanning tree such that every point vertex has even valency in the corresponding co-tree, then $\mathcal{X}$ admits an upper embedding in every orientation of triples.
\end{theorem}
In~\cite{Griggs2019}, the following two theorems were proved.
\begin{theorem}[\cite{Griggs2019}]
Every Steiner triple system admits an upper embedding in every orientation of triples.
\end{theorem}
\begin{theorem}[\cite{Griggs2019}]
Every Latin square of odd order admits an upper embedding in every orientation of triples.
\end{theorem}
The structure of the rest of this paper is as follows. In Section~\ref{sec:nonembed}, using Theorem~\ref{thm:jungerman}, we exhibit a symmetric configuration on 21 points which does not admit an upper embedding in any orientation of triples. We further show that such a symmetric configuration $v_3$ exists for all $v\geq 21$ where $v$ is odd. Thus the situation for symmetric configurations is in stark contrast to that for Steiner triple systems and Latin squares of odd order.\cite{Grannell2005,Griggs2018} The $21_3$ configuration is in fact the smallest example of a non-upper embeddable configuration, and this is dealt with in Section~\ref{sec:small}. The results in this section require extensive computer calculations, and as a by-product we needed to extend the table of configurations $v_3, 7\leq v\leq 18$, given in~\cite{Betten2000}, to the case where $v=19$. This is Section~\ref{sec:enum}. Finally, in Section~\ref{sec:everyorient} we exhibit an infinite class of symmetric configurations which are upper embeddable in every orientation of triples.
% ================================================================
\section{Non-embeddable configurations}\label{sec:nonembed}
First, we use Theorem~\ref{thm:jungerman} to deduce the existence of a symmetric configuration which is not upper embeddable in any orientation.
\begin{theorem}\label{thm:bad21}
There exists a symmetric configuration $21_3$ which is not upper embeddable in any orientation.
\end{theorem}

\begin{proof}
We construct the configuration via its Levi graph. Let $H_1,H_2,H_3$ be three copies of the Heawood graph, which is the Levi graph of the unique $7_3$ configuration. Delete one edge from each of $H_1,H_2,H_3$ (the Heawood graph is arc-transitive so the choice of edges is arbitrary). Join the point (black) vertex of valency 2 in $H_1$ so obtained to the block (white) vertex of valency 2 in $H_2$, and similarly between $H_2,H_3$ and $H_3,H_1$. The graph $G$ thus created is a cubic bipartite graph of girth 6 (Figure~\ref{fig:bad21}) and so is the Levi graph of a $21_3$ symmetric configuration $\mathcal{X}$.

Let $T$ be a spanning tree of $G$, and let $T_1,T_2,T_3$ be the subgraphs of $T$ induced by the vertices of $H_1,H_1,H_3$ respectively. It is easy to see that exactly one of the following two situations arises.
\begin{enumerate}[label=(\roman*),itemsep=0ex,topsep=0ex]
	\item $T$ contains exactly two of the newly added edges between $H_1,H_2,H_3$; and all of $T_1,T_2,T_3$ are connected.
	\item $T$ contains all three of the newly added edges between $H_1,H_2,H_3$; and exactly one of $T_1,T_2,T_3$ is disconnected.
\end{enumerate}
In case (i), without loss of generality we may assume that $T$ contains the edges between $H_1,H_2$ and $H_1,H_3$. Now consider the subgraph of the co-tree of $T$ induced by the vertices of $H_1$. This is disconnected from the remainder of the co-tree, and contains exactly 7 edges. Thus the co-tree of $T$ contains a component with an odd number of edges.

In case (ii), without loss of generality we may assume that $T_2$ is disconnected, and so $T_1$ is connected and the same analysis as in case (i) shows that the co-tree of $T$ contains a component with an odd number of edges.

Thus by Theorem~\ref{thm:jungerman}, $\mathcal{X}$ is not upper embeddable in any orientation.
\end{proof}
We note that this configuration has appeared before in the literature in a different context. Viewing the configuration as a 3-regular, 3-uniform hypergraph, Bollob\'as and Harris~\cite{Bollobas1985} constructed it as an example of a hypergraph with chromatic number 3. The underlying method of the construction in Theorem~\ref{thm:bad21} has also appeared in the context of geometric configurations; Gr\"unbaum uses a similar technique in~\cite[Fig. 2.3.3]{Grunbaum2009} to construct a combinatorial configuration $16_3$ which is not geometrically realizable.

\begin{figure}[h]\centering
\begin{tikzpicture}[x=0.3mm,y=-0.3mm,inner sep=0.2mm,scale=0.6,thick,wvertex/.style={circle,draw,minimum size=10,fill=white},bvertex/.style={circle,draw,minimum size=10,fill=black}]
\node at (385,230) (H1) {$H_1$};
\node at (200,500) (H2) {$H_2$};
\node at (570,500) (H3) {$H_3$};
\node at (204.3,609.7) [bvertex] (v1) {};
\node at (154.5,597.3) [wvertex] (v2) {};
\node at (115,564.4) [bvertex] (v3) {};
\node at (209,379) [wvertex] (v4) {};
\node at (158.8,389.4) [bvertex] (v5) {};
\node at (117.9,420.6) [wvertex] (v6) {};
\node at (94.7,466.4) [bvertex] (v7) {};
\node at (93.6,517.7) [wvertex] (v8) {};
\node at (295.4,568.1) [bvertex] (v9) {};
\node at (318.6,522.3) [wvertex] (v10) {};
\node at (319.7,470) [bvertex] (v11) {};
\node at (298.3,424.2) [wvertex] (v12) {};
\node at (258.8,391.5) [bvertex] (v13) {};
\node at (254.7,599.3) [wvertex] (v14) {};
\node at (453,316.2) [bvertex] (v15) {};
\node at (406.9,338.7) [wvertex] (v16) {};
\node at (355.5,339) [bvertex] (v17) {};
\node at (308.3,136.5) [wvertex] (v18) {};
\node at (276.4,176.8) [bvertex] (v19) {};
\node at (265.3,226.9) [wvertex] (v20) {};
\node at (277,276.9) [bvertex] (v21) {};
\node at (309.2,316.9) [wvertex] (v22) {};
\node at (496,225.8) [bvertex] (v23) {};
\node at (484.3,175.8) [wvertex] (v24) {};
\node at (452.1,135.8) [bvertex] (v25) {};
\node at (405.9,113.7) [wvertex] (v26) {};
\node at (354.4,114) [bvertex] (v27) {};
\node at (484.9,275.8) [wvertex] (v28) {};
\node at (678.8,513.7) [bvertex] (v29) {};
\node at (658,560.5) [wvertex] (v30) {};
\node at (618.8,593.8) [bvertex] (v31) {};
\node at (452.5,470) [wvertex] (v32) {};
\node at (453.9,520.4) [bvertex] (v33) {};
\node at (477.6,565.9) [wvertex] (v34) {};
\node at (518.7,596.7) [bvertex] (v35) {};
\node at (569.1,606.6) [wvertex] (v36) {};
\node at (653.7,416.8) [bvertex] (v37) {};
\node at (612.6,386) [wvertex] (v38) {};
\node at (562.2,376.1) [bvertex] (v39) {};
\node at (512.6,388.8) [wvertex] (v40) {};
\node at (473.3,422.2) [bvertex] (v41) {};
\node at (677.3,462.2) [wvertex] (v42) {};
\path
(v1) edge (v2)
(v1) edge (v14)
(v2) edge (v3)
(v3) edge (v8)
(v4) edge (v5)
(v4) edge (v13)
(v5) edge (v6)
(v6) edge (v7)
(v7) edge (v8)
(v9) edge (v10)
(v9) edge (v14)
(v10) edge (v11)
(v12) edge (v13)
(v7) edge (v14)
(v2) edge (v5)
(v3) edge (v10)
(v6) edge (v11)
(v4) edge (v9)
(v1) edge (v12)
(v8) edge (v13)
(v15) edge (v16)
(v15) edge (v28)
(v17) edge (v22)
(v18) edge (v19)
(v18) edge (v27)
(v19) edge (v20)
(v20) edge (v21)
(v21) edge (v22)
(v23) edge (v24)
(v23) edge (v28)
(v24) edge (v25)
(v25) edge (v26)
(v26) edge (v27)
(v21) edge (v28)
(v16) edge (v19)
(v17) edge (v24)
(v20) edge (v25)
(v18) edge (v23)
(v15) edge (v26)
(v22) edge (v27)
(v29) edge (v30)
(v29) edge (v42)
(v30) edge (v31)
(v31) edge (v36)
(v32) edge (v33)
(v33) edge (v34)
(v34) edge (v35)
(v35) edge (v36)
(v37) edge (v38)
(v37) edge (v42)
(v38) edge (v39)
(v39) edge (v40)
(v40) edge (v41)
(v35) edge (v42)
(v30) edge (v33)
(v31) edge (v38)
(v34) edge (v39)
(v32) edge (v37)
(v29) edge (v40)
(v36) edge (v41)
(v16) edge (v41)
(v11) edge (v32)
(v12) edge (v17)
;
\end{tikzpicture}
\caption{The Levi graph of the non-embeddable configuration of Theorem~\ref{thm:bad21}}
\label{fig:bad21}
\end{figure}
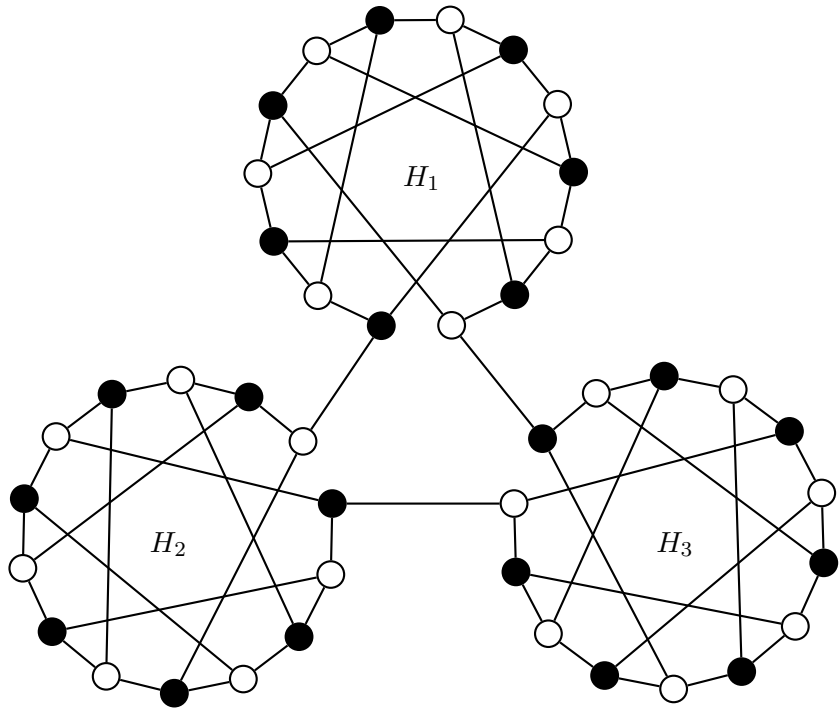

A natural extension of the argument of Theorem~\ref{thm:bad21} leads to our next result, showing that there exist symmetric configurations $v_3$ which are not upper embeddable in any orientation, for all odd $v\geq 21$.

\begin{theorem}\label{thm:badall}
Let $v$ be an odd number with $v\geq 21$. Then there exists a symmetric configuration $v_3$ which is not upper embeddable in any orientation.
\end{theorem}

\begin{proof}
Let $w=v-14$. Then $w$ is an odd number with $w\geq 7$. Choose any symmetric configuration $w_3$, and let $H_1$ be its Levi graph; let $H_2$ and $H_3$ be two copies of the Heawood graph. Delete one edge from each of $H_1$, $H_2$ and $H_3$ and carry out the same procedure as in the proof of Theorem~\ref{thm:bad21} to construct the Levi graph of a configuration $v_3$. A similar analysis to the proof of Theorem~\ref{thm:bad21} shows that the co-tree of any spanning tree of this graph contains a component with an odd number of edges, and so the configuration is not upper embeddable in any orientation.
\end{proof}

\section{Embeddability of small configurations}\label{sec:small}
In order to show that a symmetric configuration is upper embeddable in every orientation of triples, we can in principle use Theorem~\ref{thm:evenvalency} and find a spanning tree with the required property. However, as a practical method of determining upper embeddability of any given configuration, this suffers from the substantial drawback that one might need to test each spanning tree of the Levi graph. Except in trivial cases, this would be prohibitive. The following result gives a much simpler computational test.

\begin{theorem}\label{thm:covset}
Let $\mathcal{X}=(V,\mathcal{B})$ be a symmetric configuration $v_3$ for some odd $v\geq 7$, and let $G$ be its Levi graph. Suppose that there exists a subset $S$ of $V$ of size $(v-1)/2$ with the property that every block $B\in\mathcal{B}$ contains a point of $S$, and the subgraph of $G$ induced by the points of $S$ and all the blocks in $\mathcal{B}$ is connected. Then $\mathcal{X}$ is upper embeddable in every orientation.
\end{theorem}
\begin{proof}
Suppose we have such a set $S$. Then the set of neighbours $N(S)$ is the whole of $\mathcal{B}$. The subgraph of $G$ induced by $S\cup\mathcal{B}$ has $(v-1)/2+v$ vertices and $3(v-1)/2$ edges, and is connected by assumption. Thus it is a tree. We may extend this subgraph to a spanning tree $T$ of $G$ by joining each of the remaining $(v+1)/2$ vertices in $V\setminus S$ to exactly one block in any way we please. Thus each vertex in $G$ corresponding to a point in $V$ has valency 1 or 3 in $T$, and hence valency 2 or 0 in the co-tree of $T$. Thus $T$ satisfies the conditions of Theorem~\ref{thm:evenvalency}.
\end{proof}

By Theorem~\ref{thm:bad21} we know that there is a configuration $21_3$ which is not upper embeddable in any orientation. The question now is whether this is the smallest such example. 

We note that by the results of~\cite[Lemma 1]{Archdeacon2004}, if $G$ is a cubic bipartite graph of order $2v$, then there exists a set of vertices of size $(v-1)/2$ in either colour class of $G$ which dominates all the vertices in the other class. Thus we know that for any symmetric configuration $v_3$ there exists a set $S$ of points of size $(v-1)/2$ such that all the blocks contain a point of $S$. However, this result does not guarantee \emph{connectedness} of the resulting subgraph of the Levi graph, which is crucial to apply Theorem~\ref{thm:covset}. Thus we are obliged to use a computational approach to test this condition for all possible configurations. Theorem~\ref{thm:covset} gives us a computationally feasible method of testing all smaller configurations, since it requires only that we test the $\displaystyle\binom{v}{\frac{v-1}{2}}$ possible subsets of $V$ of the correct size. The outline algorithm is as follows.
\begin{enumerate}[label=\arabic*.,topsep=0ex,itemsep=0ex]
	\item For given odd $v\geq 7$, generate all possible (isomorphism classes of) connected cubic bipartite graphs of order $2v$ and girth at least 6.
	\item For each graph, form each of the two configurations obtained by considering each colour class in turn as the point set. (If the graph admits an automorphism swapping the colour classes, then only a single self-dual configuration is obtained.)
	\item For each configuration, generate every possible subset of the point set $V$ of size $(v-1)/2$ until one satisfying the conditions of Theorem~\ref{thm:covset} is found.
\end{enumerate}

Generation of the graphs was done by the same method~\cite{brinkmann1996fast} as used in the previous enumeration of symmetric configurations~\cite{Betten2000}. Automorphisms were determined using the program \texttt{nauty}~\cite{mckay201494}.

Although for larger $v$ the numbers of configurations grow rapidly (see Table~\ref{tab:enum}), we were able to complete this algorithm for all odd $v$ up to and including 19. No configurations were found which fail to have a set of points satisfying the requirements of Theorem~\ref{thm:covset}. Thus we have the following result.
\begin{theorem}\label{thm:config19}
If $v$ is an odd number with $7\leq v\leq 19$, then all configurations $v_3$ are upper embeddable in every orientation.
\end{theorem}
%======================================
\section{Enumeration}\label{sec:enum}
As a by-product of the computational process, we were able to validate the published figures enumerating the configurations $v_3$, $7\leq v\leq 19$. Up to $v=18$ the enumeration was published in~\cite{Betten2000} and we agree with those figures. The enumeration was later extended to $v=19$ in~\cite{Pisanski2004}, which contains a note that the result required a large number of parallel computer runs. With the technology available at the time, this would have required manual collation of many hundreds of results. Our runs revealed a discrepancy with the figure for $v=19$, and following independent checking~\cite{Brinkmann2018} we give a revised table in Table~\ref{tab:enum}.

For completeness we describe here the properties enumerated in Table~\ref{tab:enum}, following the notation of~\cite{Betten2000}. For a configuration $\mathcal{X}$, an \emph{automorphism} is a permutation of the points and blocks of $\mathcal{X}$ which preserves incidence. The \emph{dual} of $\mathcal{X}$ is the configuration obtained by reversing the roles of the points and blocks of $\mathcal{X}$. If $\mathcal{X}$ is isomorphic to its dual, we say it is \emph{self-dual}, and the isomorphism between $\mathcal{X}$ and its dual is an \emph{anti-automorphism}. An anti-automorphism of $\mathcal{X}$ of order two is called a \emph{polarity}, and a configuration admitting such an isomorphism is \emph{self-polar}. The group of all automorphisms of $\mathcal{X}$ (preserving the roles of points and blocks) is denoted by $\Aut(\mathcal{X})$, and the group of all automorphisms and anti-automorphisms by $A(\mathcal{X})$. If $\Aut(\mathcal{X})$ acts transitively on the points of $\mathcal{X}$ then we say $\mathcal{X}$ is \emph{point-transitive}. A \emph{flag} of $\mathcal{X}$ is an ordered pair $(p,B)$ with $p\in B$; if $\Aut(\mathcal{X})$ acts transitively on the set of flags then we say $\mathcal{X}$ is \emph{flag-transitive}; if $A(\mathcal{X})$ acts transitively on the set of flags regarded as \emph{unordered} pairs, then we say $\mathcal{X}$ is \emph{weakly flag-transitive}. A \emph{cyclic} configuration $\mathcal{X}$ is one admitting a cyclic subgroup of $\Aut(\mathcal{X})$ acting regularly on points. A \emph{blocking set} $X$ is a subset of points of $\mathcal{X}$ such that each block of $\mathcal{X}$ contains at least one element of $X$ and one element not in $X$. A \emph{blocking set-free} configuration contains no blocking sets.
\begin{table}[h]
\centering\setlength{\tabcolsep}{1.1em}
\begin{tabular}{rrrrrrrrrr}
	\hline
	$v$ & $a$ & $b$ & $c$ & $d$ & $e$ & $f$ & $g$ & $h$ & $i$ \\
	\hline
	7 & 1 & 1 & 1 & 1 & 1 & 1 & 1 & 1 & 0 \\ 
	8 & 1 & 1 & 1 & 1 & 1 & 1 & 1 & 0 & 0 \\ 
	9 & 3 & 3 & 3 & 2 & 1 & 1 & 1 & 0 & 0 \\ 
	10 & 10 & 10 & 10 & 2 & 1 & 1 & 1 & 0 & 0 \\ 
	11 & 31 & 25 & 25 & 1 & 1 & 0 & 0 & 0 & 0 \\ 
	12 & 229 & 95 & 95 & 4 & 3 & 1 & 1 & 0 & 0 \\ 
	13 & 2,036 & 366 & 365 & 2 & 2 & 1 & 1 & 1 & 0 \\ 
	14 & 21,399 & 1,433 & 1,432 & 3 & 3 & 1 & 1 & 0 & 1 \\ 
	15 & 245,342 & 5,802 & 5,799 & 5 & 4 & 1 & 1 & 0 & 1 \\ 
	16 & 3,004,881 & 24,105 & 24,092 & 6 & 4 & 2 & 2 & 0 & 4 \\ 
	17 & 38,904,499 & 102,479 & 102,413 & 2 & 2 & 0 & 0 & 0 & 13 \\ 
	18 & 530,452,205 & 445,577 & 445,363 & 9 & 5 & 1 & 1 & 0 & 47 \\ 
	19 & 7,597,040,188 & 1,979,772 & 1,979,048 & 3 & 3 & 1 & 1 & 4 & 290 \\ 
	\hline
	\multicolumn{10}{p{15cm}}{\footnotesize
		\textbf{Note}: $a$ is the number of configurations $v_3$; $b$ is the number of self-dual configurations; $c$ is the number of self-polar configurations; $d$ is the number of point-transitive configurations; $e$ is the number of cyclic configurations; $f$ is the number of flag-transitive configurations; $g$ is the number of weakly flag-transitive configurations; $h$ is the number of connected blocking set-free configurations; $i$ is the number of disconnected configurations.
	} 
\end{tabular}
\caption{Numbers of configurations $v_3$}
\label{tab:enum}
\end{table}

The table in~\cite{Betten2000} included a count for weakly flag-transitive configurations $v_3$. For consistency we include this in Table~\ref{tab:enum}, but note that this count is not required in the case of symmetric configurations $v_3$ as the following shows.
\begin{proposition}
Let $\mathcal{X}$ be a weakly flag-transitive symmetric configuration $v_3$. Then $\mathcal{X}$ is flag-transitive.
\end{proposition}
\begin{proof}
Let $G$ be the Levi graph of $\mathcal{X}$. Then by the definition of weak flag-transitivity, $G$ is edge-transitive. The following standard results are well known; see for example~\cite[Lemmas 3.2.1, 3.2.2]{godsil2001algebraic}:
\begin{enumerate}[topsep=0ex,itemsep=0ex,label=(\roman*)]
	\item If $G$ is not vertex-transitive, then its automorphism group has precisely two vertex orbits, corresponding to the points and blocks respectively of $\mathcal{X}$;
	\item if $G$ is vertex-transitive, then it is arc-transitive since its valency is odd.
\end{enumerate}
In case (i), $\mathcal{X}$ is not self-dual and hence any automorphism mapping an unordered pair $\{p,B\}$ to $\{p',B'\}$ in fact maps the ordered pair $(p,B)$ to $(p',B')$. In case (ii), if its Levi graph is arc-transitive then certainly $\mathcal{X}$ is flag-transitive.
\end{proof}

%===========================
\section{Configurations embeddable in every orientation}\label{sec:everyorient}
So far we have shown that every configuration $v_3$ where $v$ is odd and $7\leq v\leq 19$ is upper embeddable in every orientation (Theorem~\ref{thm:config19}).
Further, for $v$ odd and $v\geq 21$, we have exhibited a configuration which is not upper embeddable in any orientation (Theorem~\ref{thm:badall}). We are of the opinion that for $v$ odd and $v\geq 21$, in some sense nearly all configurations $v_3$ are upper embeddable in every orientation. Whilst we are unable to prove this, the next result shows that there exists an infinite family of such configurations.
\begin{theorem}\label{thm:cyclic}
	Let $v\geq 7$ be an odd number, and let $\mathcal{X}$ be the cyclic configuration $v_3$ generated by the block $\{0,1,3\}$. Then $\mathcal{X}$ is upper embeddable in every orientation.
\end{theorem}

\begin{proof}
	It suffices to exhibit a set $S$ of the points of $\mathcal{X}$ of size $(v-1)/2$ such that all blocks of $\mathcal{X}$ contain a point of $S$, and the subgraph of the Levi graph of $C$ induced by the points of $S$ and all the blocks of $\mathcal{X}$ is connected. There are two cases to consider.
	
	\textbf{Case 1}: $v\equiv 1 \pmod 4$. We take $S$ to be the set of points $\{4i,4i+1:0\leq i\leq (v-5)/4\}$, so $S=\{0,1,4,5,\ldots,v-5,v-4\}$. All the blocks of $\mathcal{X}$ have the form $\{m,m+1,m+3\}$ (mod $v$) and it is easy to see that all such blocks meet $S$ in at least one point. To show connectedness, we must find a path from any point of $S$ to any other. For $0\leq i\leq (v-5)/4$, there is a path between the points $4i$ and $4i+1$ via the block $\{4i,4i+1,4i+3\}$. For $0\leq i\leq (v-9)/4$, there is a path between points $4i+1$ and $4(i+1)$ via the block $\{4i+1,4i+2,4i+4\}$. So any two points of $S$ are joined by some path.
	
	\textbf{Case 2}: $v\equiv 3 \pmod 4$. We take $S$ to be the set $\{2i:0\leq i\leq (v-1)/2,i\neq (v-3)/2\}$, so $S=\{0,2,4,\ldots,v-7,v-5,v-1\}$. As before, it is easy to see that all blocks of $C$ contain a point of $S$. Point 0 is connected to points 2 and $v-1$ via the block $\{v-1,0,2\}$. For $1\leq i\leq (v-7)/2$, point $2i$ is connected to point $2(i+1)$ via the block $\{2i-1,2i,2i+2\}$. So any two points of $S$ are joined by some path.
\end{proof}

A method of constructing symmetric configurations $v_3$ dates back to the paper by Martinetti from 1887~\cite{martinetti1887sulle}. From a given $v_3$ configuration choose two disjoint triples $\{x_1,x_2,x_3\}$ and $\{y_1,y_2,y_3\}$. At least one of the pairs $\{x_1,y_1\}$, $\{x_1,y_2\}$ or $\{x_1,y_3\}$, say without loss of generality the first of these, does not appear in any triple. Remove the above two triples and introduce a new point $z$ and three new triples $\{x_1,y_1,z\}$, $\{x_2,x_3,z\}$ and $\{y_2,y_3,z\}$ to obtain a $(v+1)_3$ configuration. Configurations which can be constructed in this way from smaller configurations are called \emph{reducible}; others are \emph{irreducible}. Only relatively recently have all connected irreducible configurations $v_3$ been determined~\cite{Boben2007}. When $v$ is odd these are precisely the cyclic configurations of Theorem~\ref{thm:cyclic} as well as the Pappus configuration (one of the three $9_3$ configurations). 
Thus the following corollary is immediate.
\begin{corollary}
	Let $v$ be an odd number with $v\geq 7$. Then any irreducible configuration $v_3$ is upper embeddable in every orientation of triples.
\end{corollary}

% ================================================================
\section{Open questions}\label{sec:questions}
We close with some intriguing open questions which these investigations raise. We know that the smallest example of a configuration $v_3$ which is not upper embeddable arises at $v=21$. However we do not know whether this example is unique at $v=21$. We also know from Theorem~\ref{thm:badall} that we may construct arbitrarily large configurations which are not upper embeddable in any orientation. These observations suggest the first question.

\textbf{Question 1}: Is there a unique configuration $21_3$ which is not upper embeddable in any orientation? More generally, is the ``stitching'' construction of Theorems~\ref{thm:bad21} and~\ref{thm:badall} essentially the only way to construct a connected configuration which is not upper embeddable?

A notable feature of our investigations is that the only examples of symmetric configurations $v_3$ which fail to be upper embeddable in \emph{every} orientation are not in fact upper embeddable in \emph{any} orientation. This prompts the following question, which is probably the most fundamental one to arise from this work.

\textbf{Question 2}: If a symmetric configuration $v_3$ is upper embeddable in \emph{one} orientation, is it then necessarily upper embeddable in \emph{every} orientation?

Looking at the Levi graph of the non-embeddable configuration (Figure~\ref{fig:bad21}), it is apparent that it has a very special form. In particular, its edge and vertex connectivity are both 2. One possible avenue for future research would be to try to impose some condition on the Levi graph which would guarantee upper embeddability of the associated configuration.

\textbf{Question 3}: Is there a simple graph-theoretical condition on the Levi graph of a symmetric configuration $v_3$ which would guarantee upper embeddability of the configuration?

% ================================================================
\section{Acknowledgements}
The third author acknowledges support from the APVV Research Grants 15-0220 and 17-0428, and the VEGA Research Grants 1/0142/17 and 1/0238/19.

The authors are grateful to G. Brinkmann and T. Pisanski for assistance in verifying the count of configurations on 19 points.

% ================================================================
\bibliographystyle{abbrv}

\end{document}